\documentclass[preprint]{elsarticle}

\usepackage{nccmath}
\usepackage{amsmath}
\usepackage{amsthm}
\usepackage{amssymb}
\usepackage{pdfpages}
\usepackage{afterpage}
\usepackage{bm}

\theoremstyle{plain}
\newtheorem{thm}{Theorem}[section]
\newtheorem{lem}{Lemma}[section]
\newtheorem{prop}{Proposition}[section]

\newtheorem{con}{Conjecture}[section]

\theoremstyle{definition}
\newtheorem{rem}{Remark}[section]

\newcommand{\BC}{\mathbb{C}}

\newcommand{\VEC}{\mathrm{vec}}
 
\newcommand{\Tr}{\mathrm{tr}} 

\newcommand{\Diag}{\mathrm{diag}}

\begin{document}

\begin{frontmatter}

\title{A Generalization of the B\"{o}ttcher-Wenzel inequality for three rectangular matrices}

\author{Motoyuki NOBORI}
\address{Graduate School of Science and Engineering, Ehime University, 2-5 Bunkyo-cho, Matsuyama 790-8577, Japan}
\ead{m819002c@mails.cc.ehime-u.ac.jp}

\begin{abstract}
Let $m,n$ be positive integers. For all $m\times n$ complex matrices $A, C$ and an $n\times m$ matrix $B$, we define a generalized commutator as $ABC-CBA$. We estimate the Frobenius norm of it, and finally get the inequality, which is a generalization of the B\"{o}ttcher-Wenzel inequality. If $n=1$ or $m=1$, then the Frobenius norm of $ABC-CBA$ can be estimated with a tighter upper bound.
\end{abstract}


\end{frontmatter}

\section{Introduction}
Let $n,m$ be positive integers. The B\"{o}ttcher-Wenzel inequality is an estimate on the Frobenius norm of the commutator $AB-BA$, where both $A$ and $B$ are either real or complex $n\times n$ matrices. It asserts that 
\begin{align}\label{ineq:bwineq}
\|AB-BA\|_{F}^{2}\le 2\|A\|^{2}_{F}\|B\|^{2}_{F},
\end{align}
where $\|A\|_{F}=\sqrt{\Tr(AA^{*})}$ is the Frobenius norm, and $A^{*}=\overline{A}^{T}$ is the conjugate transpose.
Regarding (\ref{ineq:bwineq}), several proofs have been discovered so far (cf. \cite{TheFro, OnSome, VarBou} for the real case, and \cite{ProOf,NorSca} for the complex case). It is also known that the Frobenius norm of the commutator can be bounded above by an even smaller quantity than that given in (\ref{ineq:bwineq}). For example,
\begin{align}\label{ineq:sharpenbwineq1}
\|AB-BA\|_{F}^{2}\le 2\|A\|^{2}_{(2),2}\|B\|^{2}_{F}
\end{align}
holds (cf. \cite{OnSome, VarBou}), where $\|A\|_{(2),2}=\sqrt{\sigma^{2}_{1}(A)+\sigma^{2}_{2}(A)}$ is the Ky Fan (2,2)-norm and $\sigma_{i}(A)$ is the $i$-th largest singular value of $A$ for $i\in \{1 ,\cdots,n\}.$ Also,
\begin{align}\label{ineq:sharpenbwineq2}
\|AB-BA\|_{F}^{2}\le \|A\otimes B-B\otimes A\|^{2}_{F} =2(\|A\|^{2}_{F}\|B\|_{F}^{2}-|(A,B)|^{2})
\end{align}
holds \cite{TheFro}, where $(A,B)=\Tr(AB^{*})$ denotes the inner product of $A$ and $B$. In recent years, there have been attempts to generalize the B\"{o}ttcher-Wenzel inequality by extending the notions of commutators and norms(cf. \cite{GBWinnerpro, GBWqdefor}). Although different from the original B\"{o}ttcher-Wenzel inequality, estimates for the Frobenius norm of a matrix consisting of the product and difference of three real square matrices has been also analyzed \cite{ANor}.\par
In this paper, following \cite{ANor}, we define the generalized commutator as $ABC-CBA$ for $m\times n$ matrices $A,C,$ and an $n\times m$ matrix $B$, and we estimate its Frobenius norm. We finally obtain the following result.
\begin{thm}\label{thm:gbwineq1}
For all $m\times n$ complex matrices $A,C,$ and $n\times m$ complex matrix $B$, 
\begin{align}\label{ineq:gbwineq1}
\|ABC-CBA\|_{F}^{2} \le 2\|C\|_{2}^{2}\|A\|^{2}_{(2),2}\|B\|_{F}^{2}
\end{align}
where $\|A\|_{2}=\sigma_{1}(A)$ is the spectral norm of $A$. Furthermore, if $m=1$ or $n=1$, then
\begin{align}\label{ineq:sharpengbwineq1}
\|ABC-CBA\|_{F}^{2} \le \|C\|_{2}^{2}\|A\|^{2}_{(2),2}\|B\|_{F}^{2}
\end{align}
holds.
\end{thm}
(\ref{ineq:gbwineq1}) is a generalization of (\ref{ineq:sharpenbwineq1}), because when $m=n$ and $C$ is the $n\times n$ identity matrix $I_{n}$, the left- and right-hand sides of (\ref{ineq:gbwineq1}) coincide with those of (\ref{ineq:sharpenbwineq1}), due to the fact that $\|C\|_{2}=1$ and $ABC-CBA=AB-BA$. Moreover, from (\ref{ineq:gbwineq1}) and the fact that $\|A\|_{(2),2}\le \|A\|_{F}$, it follows that 
\begin{align*}
  \|ABC-CBA\|^{2}_{F}\le 2\|C\|_{2}^{2}\|A\|^{2}_{F}\|B\|^{2}_{F},
\end{align*}
which is a generalization of (\ref{ineq:bwineq}). Regarding (\ref{ineq:sharpengbwineq1}), note that $\|Z\|_{2}=\|Z\|_{F}=\|Z\|_{(2),2}$ for any $m\times n$ complex matrice $Z$ when $n=1$ or $m=1$. Since the right-hand side of (\ref{ineq:sharpengbwineq1}) satisfies $\|C\|_{2}^{2}\|A\|^{2}_{(2),2}\|B\|_{F}^{2}\le 2\|C\|_{2}^{2}\|A\|^{2}_{(2),2}\|B\|_{F}^{2}$, it follows that the upper bound in (\ref{ineq:sharpengbwineq1}) is tighter than the upper bound in (\ref{ineq:gbwineq1}).
\par The structure of this paper is as follows. In Section 2, we prove Theorem \ref{thm:gbwineq1}. In Section3, we examine whether there exists an inequality for the generalized commutator $ABC-CBA$ that may be regarded as a generalization of the B\"{o}ttcher-Wenzel inequality in a direction different from Theorem \ref{thm:gbwineq1}. \par
Before proceeding to the next Section, we briefly introduce the additional notation used in this paper. We denote by $M_{m,n}(\BC)$ a vector space on the set of all $m\times n$ complex matrices with the inner product $(A,B)= \Tr(AB^{*})$. Define $O_{m\times n} \in M_{m,n}(\BC)$ as the zero matrix of size $m\times n$. Especially, we shall write $M_{n,n}(\BC)$, $M_{n,1}(\BC)$, $O_{n\times n}$, and $O_{n\times 1}$ simply as $M_{n}(\BC)$, $\BC^{n}$, $O_{n}$, and $\bm{0}_{n}$, respectively. Let $k,l$ be positive integers. Then, for any $k\times l$ matrix $A=[a_{i,j}]_{1\le i\le k, 1\le j\le l}$ and $m\times n$ matrix $B$, the $km\times ln$ matrix $A\otimes B=[a_{i,j}B]_{1\le i\le k, 1\le j\le l}$ is defined as the Kronecker product of $A$ and $B$. Suppose $A$ is an $n\times n$ square matrix. Then, we denote by $\Lambda(A)$ the multiset of eigenvalues of $A$, counted with algebraic multiplicities. Moreover, if the definition is valid, then for all $i\in \{1,\cdots ,n\},$ let $\lambda_{i}(A)\in \Lambda(A)$ denote the $i$-th largest eigenvalue of $A$. When $A$ and $B$ are Hermitian, the notation $A\ge B$ means that $A-B$ is positive semidefinite. We write $\Diag\{d_{1},\cdots ,d_{n}\}$ for the $n\times n$ diagonal matrix with diagonal elements $d_{i}\in \BC (1\le i\le n)$. For matrices $A_{i}\in M_{k_{i},l_{i}}(\BC) (1\le i\le n),$ we use $\Diag\{A_{1},\cdots , A_{n}\}$ to denote the $(\Sigma_{i=1}^{n}k_{i})\times (\Sigma_{i=1}^{n}l_{i})$ block diagonal matrix with $A_{1},\cdots ,A_{n}$ on the diagonal.

\section{A generalization of the B\"{o}ttcher-Wenzel inequality}
Let $A,C\in M_{m,n}(\BC)$ and $B\in M_{n,m}(\BC)$.\ If at least one of $A,B$ or $C$ is the zero matrix, then $ABC-CBA=O_{m\times n}$, and the inequality (\ref{ineq:gbwineq1}) holds. Henceforth, we consider the case where $A,C\ne O_{m\times n}$ and $B\ne O_{n\times m}$. To begin with, we consider the case when $n=1$ or $m=1$. We prove the inequality (\ref{ineq:sharpengbwineq1}) in the following Theorem.
\begin{thm}\label{thm:bwineqm1}
If $n=1$ or $m=1$, then 
\begin{align}\notag
\|ABC-CBA\|^{2}_{F} \le \|C\|^{2}_{2}\|A\|^{2}_{(2),2}\|B\|^{2}_{F}.
\end{align}
\end{thm}
\begin{proof} 
We prove the statement only for the case $n=1$. The case $m=1$ follows analogously.
Let $C =U_{1}DU_{2}$ be the singular value decomposition of $C$, where $U_{1}$ is an $m\times m$ unitary matrix, $D=[\sigma_{1}(C), 0, \cdots,0]^{T}$, and $U_{2}$ is an $1\times 1$ unitary matrix. Put $X=[x_{1},\cdots ,x_{m}]=U^{*}_{2}BU_{1}^{*}$, $Y=[y_{1}, \cdots, y_{m}]^{T}=U_{1}AU_{2}.$ Since the Frobenius norm is unitarily invariant, it follows that
\begin{align}\notag
\|ABC-CBA\|^{2}_{F}&=\|CBA-ABC\|^{2}_{F}\\ \notag
&=\|DXY-YXD\|^{2}_{F}\\ \notag
&=\|\sigma_{1}(C)(\begin{bmatrix}
  \Sigma_{i=1}^{m}x_{i}y_{i}\\
  0\\
  \vdots \\
  0
\end{bmatrix} -
\begin{bmatrix}
  x_{1}y_{1}\\
  x_{1}y_{2}\\
  \vdots \\
  x_{1}y_{m}
\end{bmatrix})\|_{F}^{2}\\ \notag
&=\sigma^{2}_{1}(C)(|\Sigma_{i=2}^{m}x_{i}y_{i}|^{2}+\Sigma_{i=2}^{m}|x_{1}y_{i}|^{2}).
\end{align}
Here, since $|\Sigma_{i=2}^{m}x_{i}y_{i}|^{2}\le \Sigma_{2\le i,j\le m}|x_{i}|^{2}|y_{j}|^{2}$, we have 
\begin{align}\notag
\sigma^{2}_{1}(C)(|\Sigma_{i=2}^{m}x_{i}y_{i}|^{2}+\Sigma_{i=2}^{m}|x_{1}y_{i}|^{2})
&\le \sigma^{2}_{1}(C)(\Sigma_{2\le i,j\le m}^{}|x_{i}|^{2}|y_{j}|^{2} +|x_{1}|^{2}\Sigma_{i=2}^{m}|y_{i}|^{2})\\ \notag
&= \sigma^{2}_{1}(C)\|X\|^{2}_{F} \Sigma_{i=2}^{m}|y_{i}|^{2} \\ \notag
&\le \sigma^{2}_{1}(C)\|Y\|_{F}^{2}\|X\|^{2}_{F} =\|C\|_{2}^{2}\|A\|_{(2),2}^{2}\|B\|_{F}^{2}. \qedhere
\end{align}
\end{proof}
In fact, when $n=1$ or $m=1$, the Frobenius norm of $ABC-CBA$ has an upper bound tighter than the right-hand side of (\ref{ineq:sharpengbwineq1}).
\begin{prop}\label{prop:tightehrubofgm1}
If $m=1$ or $n=1$, then
\begin{align*}
  \|ABC-CBA\|_{F}^{2}\le \frac{\|B\|_{F}^{2}}{2}\|A\otimes C-C\otimes A\|^{2}_{F}.
\end{align*}
\end{prop}
\begin{proof}
See the proof of Proposition \ref{prop:gstbw}.
\end{proof}
From now on, we assume that $m, n\ge 2$. In the following discussion, our argument is guided by Section 2 and 3 of \cite{OnSome}.\ We simply write the generalized commutator $ABC-CBA$ as $[A,C]_{B}$. 
\begin{lem}\label{lem:propsofgencom}
A generalized commutator $[A,C]_{B}$ has the following properties.
\begin{description}
  \item{\textup{($\mathrm{i}$)}} $[A,C]_{B}=-[C,A]_{B}$ for all $A,C\in M_{m,n}(\BC)$ and $B\in M_{n,m}(\BC)$.
  \item{\textup{($\mathrm{ii}$)}} $([A,C]_{B})^{*}=-[A^{*},C^{*}]_{B^{*}}$ for all $A,C\in M_{m,n}(\BC)$ and $B\in M_{n,m}(\BC)$.
  \item{\textup{($\mathrm{iii}$)}} $[A,C]_{\alpha B_{1}+\beta B_{2}}=\alpha[A,C]_{B_{1}}+\beta[A,C]_{B_{2}}$ for all $\alpha ,\beta \in \BC$, $A,C\in M_{m,n}(\BC)$, and $B_{1},\ B_{2}\in M_{n,m}(\BC)$
\end{description}
\end{lem}
We denote by $\VEC$ the vectorization map that stacks the columns of a matrix into a column vector. For example, for $B=[\bm{b}_{1}, \cdots ,\bm{b}_{m}] \in M_{n,m}(\BC)$, 
\begin{align}\notag
\VEC:M_{n,m}(\BC)\to \BC^{mn};B\mapsto \begin{bmatrix}
  \bm{b}_{1}\\
  \vdots \\
  \bm{b}_{m}
\end{bmatrix}.
\end{align}
\begin{lem}[\cite{Matrix}]\label{lem:popsofvec}
The map $\VEC$ has the following properties.
\begin{description}
  \item{\textup{($\mathrm{i}$)}} $\VEC$ is  linear and bijective.
  \item{\textup{($\mathrm{ii}$)}} $(\VEC(X),\VEC(Y))=(X,Y)$ for all $X,Y\in M_{m,n}(\BC).$
  \item{\textup{($\mathrm{iii}$)}} $\VEC(XYZ)=(Z^{T}\otimes X)\VEC(Y)$ for all $X,Z\in M_{m,n}(\BC)$ and $Y\in M_{n,m}(\BC)$. 
\end{description}
\end{lem}
By Lemma \ref{lem:popsofvec} \textup{($\mathrm{i}$)}, $\VEC$ admits an inverse.\par

Here, we define $K_{A,C}$ to be $C^{T}\otimes A-A^{T}\otimes C$. Then, by Lemma \ref{lem:popsofvec} \textup{($\mathrm{iii}$)} and $\|B\|^{2}_{F}=\|\VEC(B)\|^{2}_{2}$, it follows that
\begin{align}\notag
\|ABC-CBA\|^{2}_{F}&=\|(C^{T}\otimes A-A^{T}\otimes C)\VEC(B)\|^{2}_{2} \\ \notag                &=\|K_{A,C}\VEC(B)\|^{2}_{2}\\ \notag
     &\le \|K_{A,C}\|^{2}_{2}\|\VEC(B)\|^{2}_{2}\\ \notag
     &=\lambda_{1}(K_{A,C}^{*}K_{A,C})\|B\|^{2}_{F}.
\end{align}
Therefore, by proving $\lambda_{1}(K_{A,C}^{*}K_{A,C})\le 2\|C\|^{2}_{2}\|A\|_{(2),2}^{2}$, we can conclude that (\ref{ineq:gbwineq1}) holds. Using an argument analogous to Proposition 2.2 in \cite{OnSome}, we obtain the following result for $K_{A,C}^{*}K_{A,C}.$
\begin{prop}\label{prop:propsofKAC}
The geometric multiplicity of every positive eigenvalue of $K_{A,C}^{*}K_{A,C}$ is at least 2.
\end{prop}
\begin{proof}
Let $\phi>0$ be an
eigenvalue of $K_{A,C}^{*}K_{A,C}$ and let $\bm{y}\in \BC^{mn}\setminus \{\bm{0}_{mn}\}$ be an eigenvector of $K_{A,C}^{*}K_{A,C}$ corresponding to $\phi$. Put $Y=\VEC^{-1}(\bm{y}) \in M_{n,m}(\BC)$. Then, by Lemma \ref{lem:popsofvec} \textup{($\mathrm{iii}$)}, we have $\phi \bm{y}=\VEC(\phi Y)=K_{A,C}^{*}K_{A,C}\VEC(Y)=\VEC([A^{*},C^{*}]_{[A,C]_{Y}})$. Since $\phi \bm{y}\neq \bm{0}_{mn}$ and the fact that $\VEC$ operator is injective, we obtain $[A^{*},C^{*}]_{[A,C]_{Y}} \neq O_{m\times n}.$ We will now show that $-K_{A,C}^{*}\VEC(Y^{*})(\neq \bm{0}_{mn})$ is orthogonal to the vector $\bm{y}$ and is also an eigenvector of $K_{A,C}^{*}K_{A,C}$. First, by Lemma \ref{lem:propsofgencom} \textup{($\mathrm{ii}$)} and Lemma \ref{lem:popsofvec} \textup{($\mathrm{ii}$)}, we see that
\begin{align}\notag
(\bm{y},-K_{A,C}^{*}\VEC(Y^{*})) &=(\bm{y},\VEC(-[A^{*},C^{*}]_{Y^{*}}))\\ \notag
&=(\bm{y},\VEC(([A,C]_{Y})^{*})) \\ \notag
&=(Y,([A,C]_{Y})^{*})\\ \notag
&=\Tr(YAYC-YAYC)=0.
\end{align}
Moreover, it follows that 
\begin{align}\notag
K_{A,C}^{*}K_{A,C}(-K_{A,C}^{*}\VEC(Y^{*}))&=-K_{A,C}^{*}(K_{A,C}K_{A,C}^{*}\VEC(Y^{*})) \\ \notag
&=-K_{A,C}^{*}\VEC([A,C]_{[A^{*},C^{*}]_{Y^{*}}}) \\ \notag
&=-K_{A,C}^{*}\VEC(([A^{*},C^{*}]_{[A,C]_{Y}})^{*})\\ \notag
&=-K_{A,C}^{*}\VEC(\phi Y^{*})\\ \notag
&=\phi(-K_{A,C}^{*}\VEC(Y^{*})).
\end{align}
Therefore, since $\bm{y}$ and $-K_{A,C}^{*}\VEC(Y^{*})$ are linearly independent eigenvectors of $K_{A,C}^{*}K_{A,C}$ corresponding to $\phi$, the geometric multiplicity of $\phi$ is more than or equal to 2.
\end{proof}

\begin{prop}[\cite{Matrix}, Weyl's inequality]\label{prop:weylanditsappl}
Let $k$ be a positive integer and $X,Y \in M_{k}(\BC)$ be Hermitian matrices.
\begin{description}
  \item{\textup{($\mathrm{i}$)}} $\lambda_{j}(X+Y)\le \lambda_{i}(X)+\lambda_{j-i+1}(Y)$ for all $i,j\in\{1,\cdots,k\}$ with $i \le j$.
  \item{\textup{($\mathrm{ii}$)}} Suppose $X\ge Y$. Then $\lambda_{i}(X)\ge \lambda_{i}(Y)$ for all $ i \in \{1,\cdots,k\}$.
\end{description}
\end{prop}
\begin{lem}[\cite{Matrix}]\label{lem:propkron}
 Let $X$ and $Y$ be positive semidefinite matrices of size $m$ and $n$, respectively. Then, $X\otimes Y$ is also positive semidefinite.
\end{lem}
The following Lemma can be proved by an argument similar to that of Theorem 3.6 in \cite{OnSome}.
\begin{lem}\label{lem:psdofKAC}
$2(\overline{C}C^{T}\otimes A^{*}A +\overline{A}A^{T}\otimes C^{*}C)\ge K_{A,C}^{*}K_{A,C}$.
\end{lem}
\begin{proof}
The result is obvious because
\begin{align}\notag
2(\overline{C}C^{T}\otimes A^{*}A +\overline{A}A^{T}\otimes C^{*}C)- K_{A,C}^{*}K_{A,C}&=(C^{T}\otimes A+A^{T}\otimes C)^{*}(C^{T}\otimes A+A^{T}\otimes C)\\ \notag
&\ge O_{mn}. \qedhere
\end{align}
\end{proof}
From the above Lemma, Proposition \ref{prop:propsofKAC}, and Proposition \ref{prop:weylanditsappl} \textup{($\mathrm{ii}$)}, we conclude that 
\begin{align}\label{neq:tipofGBW}
  \lambda_{1}(K^{*}_{A,C}K_{A,C})=\lambda_{2}(K_{A,C}^{*}K_{A,C})\le2\lambda_{2}(\overline{C}C^{T}\otimes A^{*}A+\overline{A}A^{T}\otimes C^{*}C).
\end{align}
\begin{thm}
  $\lambda_{2}(\overline{C}C^{T}\otimes A^{*}A+\overline{A}A^{T}\otimes C^{*}C) \le \sigma_{1}^{2}(C)(\sigma_{1}^{2}(A)+\sigma_{2}^{2}(A)).$
\end{thm}
\begin{proof}
  By virtue of the inequality $\sigma_{1}^{2}(C)I_{m}\ge \overline{C}C^{T}$ and Lemma \ref{lem:propkron}, it follows that $\sigma_{1}^{2}(C)I_{m}\otimes A^{*}A \ge \overline{C}C^{T}\otimes A^{*}A$, which implies that 
\begin{align*}
\sigma_{1}^{2}(C)I_{m}\otimes A^{*}A+\overline{A}A^{T}\otimes C^{*}C \ge \overline{C}C^{T}\otimes A^{*}A+\overline{A}A^{T}\otimes C^{*}C.
\end{align*}
Let $l$ be the rank of $A$ and define $D_{A}=\Diag\{\sigma_{1}^{2}(A),\cdots ,\sigma_{l}^{2}(A)\}.$ Suppose that the singular value decomposition of $\overline{A}A^{T}$ is given by $\overline{A}A^{T}=U_{1}\Diag\{D_{A},O_{m-l}\}U_{1}^{*}$. Then, by the similarity invariance of the characteristic polynomial and the fact that  $U_{1}\otimes I_{n}$ is invertible, we have
\begin{align*}
\Lambda(\sigma_{1}^{2}(C)I_{m}\otimes A^{*}A+\overline{A}A^{T}\otimes C^{*}C)=\Lambda(\sigma_{1}^{2}(C)I_{m}\otimes A^{*}A+\Diag\{D_{A},O_{m-l}\}\otimes C^{*}C).
\end{align*}
By defining $X_{i}=\sigma_{1}^{2}(C)A^{*}A+\sigma_{i}^{2}(A)C^{*}C$ for $1\le i\le l$ and $X_{i}=\sigma_{1}^{2}(C)A^{*}A$ for $l+1\le i\le m$, we can rewrite $\sigma_{1}^{2}(C)I_{m}\otimes A^{*}A+\Diag\{D_{A},O_{m-l}\}\otimes C^{*}C$ as $\Diag\{X_{1},\cdots ,X_{m}\}.$ Here, because $\sigma_{1}^{2}(A)\ge\cdots \ge \sigma_{l}^{2}(A)>0 $ and both $\sigma_{1}^{2}(C)A^{*}A$ and $C^{*}C$ are positive semidefinite, we obtain $X_{i} \ge X_{i+1}$ for every $i\in\{1,\cdots ,m-1\}$. From this and Proposition \ref{prop:weylanditsappl} \textup{($\mathrm{ii}$)}, we have $\lambda_{j}(X_{i})\ge \lambda_{j}(X_{i+1})$ for all $i,j$ with $1\le i\le m-1$ and $1\le j\le n$, which implies $\lambda_{2}(\Diag\{X_{1},\cdots ,X_{m}\})=\max\{\lambda_{1}(X_{2}),\lambda_{2}(X_{1})\}.$ In the case that $\max\{\lambda_{1}(X_{2}),\lambda_{2}(X_{1})\} = \lambda_{1}(X_{2})$, it follows from Proposition \ref{prop:weylanditsappl} \textup{($\mathrm{i}$)} that 
\begin{align*}
  \lambda_{1}(X_{2})&\le \sigma_{1}^{2}(C)\lambda_{1}(A^{*}A)+\sigma_{2}^{2}(A)\lambda_{1}(C^{*}C)\\
  &= \sigma_{1}^{2}(C)(\sigma_{1}^{2}(A)+\sigma_{2}^{2}(A)).
\end{align*}
Similarly, if $\max\{\lambda_{1}(X_{2}),\lambda_{2}(X_{1})\} = \lambda_{2}(X_{1})$, then by Proposition \ref{prop:weylanditsappl} \textup{($\mathrm{i}$)}, we have
\begin{align*}
  \lambda_{2}(X_{1})&\le \sigma_{1}^{2}(C)\lambda_{2}(A^{*}A)+\sigma_{1}^{2}(A)\lambda_{1}(C^{*}C)\\
  &= \sigma_{1}^{2}(C)(\sigma_{1}^{2}(A)+\sigma_{2}^{2}(A)).
\end{align*}
Therefore, $\max\{\lambda_{1}(X_{2}),\lambda_{2}(X_{1})\}\le \sigma_{1}^{2}(C)(\sigma_{1}^{2}(A)+\sigma_{2}^{2}(A))$, and we thus get $\lambda_{2}(\overline{C}C^{T}\otimes A^{*}A+\overline{A}A^{T}\otimes C^{*}C)\le \sigma_{1}^{2}(C)(\sigma_{1}^{2}(A)+\sigma_{2}^{2}(A)).$
\end{proof}
We conclude from the above Theorem and (\ref{neq:tipofGBW}) that 
\begin{align*}
  \lambda_{1}(K_{A,C}^{*}K_{A,C})\le 2\sigma_{1}^{2}(C)(\sigma_{1}^{2}(A)+\sigma_{2}^{2}(A)),
\end{align*}
and finally that 
\begin{align*}
  \|ABC-CBA\|_{F}^{2}\le 2\|C\|^{2}_{2}\|A\|^{2}_{(2),2}\|B\|_{F}^{2}.
\end{align*}
\begin{rem}\label{rem:oneofgen}
The following inequality can be regarded as a generalization of (\ref{ineq:sharpenbwineq2}):
\begin{align}\label{ineq:oneofgen1}
\|ABC-CBA\|_{F}^{2}\le \|C\|^{2}_{2}\|A\otimes B-B\otimes A\|^{2}_{F}. 
\end{align}
However, this does not hold in the general case. Because under the condition that $m=n=2$, $A=B=\begin{bmatrix} -1&-1\\0&0\end{bmatrix}$, and $C=\begin{bmatrix} 2&0\\0&1\end{bmatrix}$, $\|ABC-CBA\|^{2}_{F}$ becomes 1 and $\|C\|^{2}_{2}\|A\otimes B-B\otimes A\|^{2}_{F}$ becomes 0, from which it follows that $\|ABC-CBA\|_{F}^{2}> \|C\|^{2}_{2}\|A\otimes B-B\otimes A\|^{2}_{F}$.
\end{rem}

\section{Additional investigations}
Inequality (\ref{ineq:gbwineq1}) in Theorem \ref{thm:gbwineq1} is a generalization of inequality (\ref{ineq:sharpenbwineq1}) in the case where the norm of either $A$ or $C$ is the spectral norm. A generalization of the inequalities (\ref{ineq:bwineq}), (\ref{ineq:sharpenbwineq1}), and (\ref{ineq:sharpenbwineq2}) to the case where the norm of $B$ is the spectral norm is also conceivable. Based on the numerical experiments, we conjecture that the following result holds.
\begin{con}
  For all $A,C\in M_{m,n}(\BC)$ and $B \in M_{n,m}(\BC)$ with $m,n\ge2,$
  \begin{align}\label{ineq:sharpenconjbwineq1}
  \|ABC-CBA\|^{2}_{F}\le 2\|B\|^{2}_{2}\|A\|^{2}_{(2),2}\|C\|_{F}^{2}.
  \end{align}
\end{con}
Note that (\ref{ineq:sharpenconjbwineq1}) takes the same form as (\ref{ineq:sharpenbwineq1}) when $m=n$ and $B=I_{n}$. When $n=1$ or $m=1$, the inequality (\ref{ineq:sharpenconjbwineq1}) holds by Theorem \ref{thm:bwineqm1}. Under the assumption that (\ref{ineq:sharpenconjbwineq1}) is satisfied, the following inequality holds as well:
\begin{align}\label{ineq:oneofgen2}
\|ABC-CBA\|_{F}^{2}&\le \|B\|^{2}_{2}\|A\otimes C-C\otimes A\|^{2}_{F}.
\end{align}
The validity of statement (\ref{ineq:oneofgen2}) follows from (\ref{ineq:sharpenconjbwineq1}). Indeed, by (\ref{ineq:sharpenconjbwineq1}), we have $\|ABC-CBA\|^{2}_{F}\le2\|B\|^{2}_{2}\|A\|^{2}_{(2),2}\|C\|^{2}_{F}\le2\|B\|^{2}_{2}\|A\|^{2}_{F}\|C\|^{2}_{F}$, we let $C'=C-\frac{\overline{(A,C)}}{\|A\|_{F}^{2}}A$, which leads to
\begin{align}\notag
\|ABC-CBA\|_{F}^{2}&=\|ABC'-C'BA\|_{F}^{2}\\ \notag
           &\le 2\|B\|^{2}_{2}\|A\|^{2}_{F}\|C'\|^{2}_{F}\\ \notag
           &=2\|B\|^{2}_{2}(\|A\|^{2}_{F}\|C\|^{2}_{F}-|(A,C)|^{2})\\ \notag
           &=\|B\|^{2}_{2}\|A\otimes C-C\otimes A\|^{2}_{F}.
\end{align}
It is known that the following inequality holds for the Frobenius norm of $ABC-CBA$ when $A,B,$ and $C$ are real square matrices.
\begin{thm}[\cite{ANor}, Theorem 3.1]\label{thm:stbw}
For all $n\times n$ real matrices $A,B$ and $C$,
\begin{align}\label{ineq:laslzo}
\|ABC-CBA\|_{F}^{2}\le \dfrac{\|B\|^{2}_{F}}{2}\|A\otimes C-C\otimes A\|^{2}_{F}.
\end{align}
\end{thm}
The proof of Theorem \ref{thm:stbw} in \cite{ANor} relies on the existence of an $n^{2}\times n^{2}$ real skew-symmetric matrix $S$ satisfying $\|ABC-CBA\|_{F}=\|S\VEC(B)\|_{2}$, and that $S$ is a normal matrix whose eigenvalues are purely imaginary with conjugate pairs. The matrix $S$ can be constructed by permuting the rows of $K_{A,C}$. That is, there exists a permutation matrix $U$ such that $UK_{A,C}=S$. When $S$ is a complex skew-symmetric matrix, it is not necessarily normal. In other words, there is no guaranteed relationship between the Frobenius norm of $S$ and its eigenvalues. The following proposition extends Theorem \ref{thm:stbw} to the case of rectangular and complex matrices, and it also provides an alternative proof of Theorem \ref{thm:stbw}.
\begin{prop}\label{prop:gstbw}
For all $m\times n$ complex matrices $A,C,$ and $n\times m$ complex matrix $B$, 
\begin{align}\label{ineq:gstbw}
  \|ABC-CBA\|_{F}^{2}\le \frac{\|B\|^{2}_{F}}{2}\|A\otimes C-C\otimes A\|_{F}^{2}.
\end{align}
\end{prop}
\begin{proof}
It suffices to consider the case where all three matrices $A,B$, and $C$ are nonzero. We may assume $\|B\|_{F}=1$ without loss of generality. Then, since $\|ABC-CBA\|_{F}^{2}\le \sigma^{2}_{1}(K_{A,C})$ and $\|K_{A,C}\|_{F}^{2}=2(\|A\|_{F}^{2}\|C\|_{F}^{2}-|(A,C)|^{2})=\|A\otimes C-C\otimes A\|_{F}^{2}$, it is enough to show that $2\sigma_{1}^{2}(K_{A,C})\le \|K_{A,C}\|_{F}^{2}$. Proposition \ref{prop:propsofKAC} implies that 
\begin{align*}
2\sigma_{1}^{2}(K_{A,C})&=\sigma_{1}^{2}(K_{A,C})+\sigma_{2}^{2}(K_{A,C}) \\
                        &\le \Sigma_{i=1}^{mn}\sigma_{i}^{2}(K_{A,C})\\
                        &= \|K_{A,C}\|_{F}^{2},
\end{align*}
and thus the statement holds.
\end{proof}
According to the above Proposition, inequality (\ref{ineq:oneofgen2}) holds whenever $m$ and $n$ are any natural numbers and the rank of $B$ is no greater than 2, since in this case we have $\|B\|^{2}_{F}\le 2\|B\|_{2}^{2}.$
Based on (\ref{ineq:gstbw}), one might conjecture more generally that for any $m\times n$ matrices  $A,C$ and $n\times m$ matrix $B$ with rank $k$,
\begin{align}\label{ineq:wronggen}
\|ABC-CBA\|_{F}^{2}\le \dfrac{\|B\|^{2}_{F}}{k}\|A\otimes C-C\otimes A\|^{2}_{F}
\end{align}
holds. Indeed, when $m=n$ and $B=I_{n}$, $\frac{\|B\|^{2}_{F}}{n}=1$ holds, making both sides of the inequality coincide with those of (\ref{ineq:bwineq}). However, this inequality (\ref{ineq:wronggen}) does not hold in general. For example, when $n=m=3$, $A=\begin{bmatrix} 0&-1&0\\0&0&0\\0&0&0\end{bmatrix}, B=\begin{bmatrix} 1&0&0\\0&0.5&0\\0&0&0.5\end{bmatrix}$, and $C=\begin{bmatrix} -1&1&0\\1&1&0\\-1&0&0\end{bmatrix}$, we have $\|ABC-CBA\|_{F}^{2}=4.5$ and $\dfrac{\|B\|^{2}_{F}}{k}\|A\otimes C-C\otimes A\|^{2}_{F}=4$, thus $\|ABC-CBA\|_{F}^{2}> \dfrac{\|B\|^{2}_{F}}{k}\|A\otimes C-C\otimes A\|^{2}_{F}.$

\section*{Acknowledgment} 
The author used a large language model (Chat GPT, Open AI) to assist with English editing. All theoretical content was developed independently by the author. The author received no external funding for this work.
\bibliographystyle{elsarticle-num}
\bibliography{reference}
\end{document}